\documentclass{gtpart}     
\usepackage{epsfig,color} 
\usepackage{mathtools,amsmath}
\usepackage[utf8]{inputenc}
\usepackage[T1]{fontenc}

\title{Acylindrical hyperbolicity and \\ Artin-Tits groups of spherical type}

\author{Matthieu Calvez}
\givenname{Matthieu}
\surname{Calvez}
\address{Matthieu Calvez, 
Departamento de Matem\'{a}tica y Estad\'istica , Universidad de La Frontera, Francisco Salazar 1145, Temuco, Chile}
              \email{matthieu.calvez@ufrontera.cl}

\author{Bert Wiest}
\givenname{Bert}
\surname{Wiest}
\address{Bert Wiest, UFR Math\'ematiques, Universit\'e de Rennes 1, 35042 Rennes Cedex, France}
\email{bertold.wiest@univ-rennes1.fr}


\volumenumber{}
\issuenumber{}
\publicationyear{}
\papernumber{}
\startpage{}
\endpage{}
\doi{}
\MR{}
\Zbl{}
\received{}
\revised{}
\accepted{}
\published{}
\publishedonline{}
\proposed{}
\seconded{}
\corresponding{}
\editor{}
\version{}

\makeautorefname{notation}{Notation}%

\newtheorem{theorem}{Theorem}[section]
\newtheorem{lemma}[theorem]{Lemma} 
\newtheorem{proposition}[theorem]{Proposition}
\newtheorem{corollary}[theorem]{Corollary}

\theoremstyle{definition}
\newtheorem{definition}[theorem]{Definition}    
\newtheorem{remark}[theorem]{Remark}

\newtheorem{example}[theorem]{Example}

\def\N{{\mathbb N}}
\def\Z{{\mathbb Z}}

\def\wedgeR{{\wedge^{\hspace{-0.7mm}\Lsh\hspace{0.7mm}}}}
\def\veeR{{\vee^{\hspace{-0.7mm}\Lsh\hspace{0.7mm}}}}

\renewcommand{\phi}{\varphi}

\begin{document}

\begin{abstract}
We prove that, for any irreducible Artin-Tits group of spherical type~$G$, the quotient of~$G$ by its center is acylindrically hyperbolic. This is achieved by studying the additional length graph associated to the classical Garside structure on~$G$, and constructing a specific element $x_G$ of~$G/Z(G)$ whose action on the graph is loxodromic and WPD in the sense of Bestvina-Fujiwara; following Osin, this implies acylindrical hyperbolicity.  
Finally, we prove that ``generic'' elements of~$G$ act loxodromically, where the word ``generic'' can be understood in either of the two common usages: as a result of a long random walk or as a random element in a large ball in the Cayley graph.
\end{abstract}

\maketitle

\section{Introduction}

The class of acylindrically hyperbolic groups has attracted a lot of interest in the last few years: it is restrictive enough to admit several powerful theorems and wide enough to include many relevant families of groups. Osin's survey~\cite{osin} presents various equivalent definitions of acylindrical hyperbolicity and provides a lot of examples and properties. 

An isometric action of a group $G$ on a metric space $X$ is said to be \emph{acylindrical} if for every $k>0$, there exist $R,N>0$ such that 
for any two points which are at distance at least~$R$, there are at most $N$ elements of~$G$ whose action moves each of the two points by a distance of at most~$k$.
A non-virtually cyclic group $G$ is said to be \emph{acylindrically hyperbolic} (a.h.) if it admits an acylindrical action on a Gromov-hyperbolic space in which at least one element of $G$ acts loxodromically (i.e. the orbit of every point is quasi-isometric to the integers). 

Beyond the class of non-elementary hyperbolic groups, a prominent example of an a.h. group is the mapping class group of a closed surface of genus $g$ with $p$ punctures (except in a few sporadic cases): it has an acylindrical action~\cite{Bowditch08} on the curve graph, which is a Gromov-hyperbolic space~\cite{MM1}. Other examples include the outer automorphism group of a free group~\cite{bestvinafeighn}, and non-cyclic, non-directly decomposable RAAGS \cite{Koberda1,Koberda}. 

The main goal of the present paper is to add the class of irreducible Artin-Tits groups of spherical type (modulo their center) to this list. This will be achieved by studying the natural action of an irreducible Artin-Tits group of spherical type on the additional length graph associated to its classical Garside structure. 

 
Let us briefly recall what that means. 
\emph{Artin-Tits groups of spherical type} (or generalized braid groups) are a generalization of the classical Artin braid group: it was shown in \cite{BS,Del} that they enjoy a combinatorial structure analogue to that first discovered for braids by Garside \cite{Garside} (the \emph{classical Garside structure} of the braid group).  More generally, \emph{Garside groups} can be roughly defined as groups satisfying analogues of Garside's combinatorial properties \cite{DehornoyParis,DehornoyGarside,DDGKM}. A prominent feature of Garside groups is the biautomatic structure, provided by normal forms which allow to solve the word and conjugacy problems.

Motivated by the example of the classical Artin braid group (which is a close relative of the mapping class group of a punctured sphere) acting on the curve graph of the punctured sphere, the authors provided in~\cite{CalvezWiest} a combinatorial construction of the \emph{additional length graph} associated to any Garside group~$G$, equipped with a Garside structure of  finite type. 

This additional length graph $\mathcal C_{AL}(G)$ is a connected, Gromov-hyperbolic graph which admits an isometric action of~$G$. In the special case where $G$ is the Artin braid group endowed with its classical Garside structure, it was shown in~\cite{CalvezWiest} that $\mathcal C_{AL}(G)$ has infinite diameter: this was accomplished by exhibiting a family of braids which act loxodromically on $\mathcal C_{AL}$; on the other hand, periodic and reducible braids were shown to act elliptically on the additional length graph. 

It was left as an open problem in~\cite{CalvezWiest} to give conditions on the Garside group $G$ for its additional length graph to have infinite diameter. Here we give a partial answer, settling completely the case of Artin-Tits groups of spherical type (endowed with the classical Garside structure):

\begin{theorem}\label{T:InfDiam}
The additional length graph $\mathcal C_{AL}(G)$ associated to the classical Garside structure of an Artin-Tits group of spherical type $G$ has infinite diameter if and only if $G$ is irreducible.
\end{theorem}

The ``only if part" of Theorem~\ref{T:InfDiam} results from an easy computation, see Proposition~\ref{P:OnlyIf}. The strategy to show the reverse implication is quite similar to the proof of \cite[Theorem 3.5] {CalvezWiest}. It is based on the construction (Proposition~\ref{P:Lox}) of a specific element~$x_G$ of~every irreducible $G$ whose action on the graph $\mathcal C_{AL}(G)$ is loxodromic. 

The center of an irreducible Artin-Tits group of spherical type is known to be cyclic, generated by a distinguished element $\Omega$. The element $\Omega$ turns out to act trivially on $\mathcal C_{AL}(G)$, hence the action of $G$ descends to an action of $G/Z(G)$. In this context we shall prove:

\begin{theorem}\label{T:WPD}
For each irreducible Artin-Tits group of spherical type endowed with its classical Garside structure, the element $x_G\in G/Z(G)$ acts WPD on the additional length graph $\mathcal C_{AL}(G)$.
\end{theorem}

Here, the property referred to is the so-called \emph{Weak Proper Discontinuity} introduced by Bestvina-Fujiwara in~\cite{BestvinaFujiwara02} (we refer to Section~\ref{S:Acylindrical} for the definition). In other words, as $\mathcal C_{AL}(G)$ is Gromov-hyperbolic, Theorem~\ref{T:WPD} says that $G/Z(G)$ satisfies Property (AH3) of Theorem 1.2 in Osin's paper~\cite{osin}. This can be restated in the following way: 

\begin{theorem}\label{T:Acylindrical}
If $G$ is an irreducible Artin-Tits group of spherical type, then $G/Z(G)$ is acylindrically hyperbolic. 
\end{theorem} 

Finally, we deduce some consequences regarding the genericity of loxodromically-acting elements in an irreducible Artin-Tits group of spherical type $G$. There are two commonly used ways of  ``picking an element of~$G$ at random''. The first is to make a long random walk in the Cayley graph of~$G$, the second is to pick a random point from a large ball centered on the neutral element in the Cayley graph of~$G$. We prove that with either of these two notions, the probability of picking an element which acts loxodromically on~$\mathcal C_{AL}$ tends to~$1$ exponentially quickly as the length of the random walk, or the size of the ball respectively, tends to infinity. The main tools used in these proofs come from \cite{SistoGeneric} in the random-walk setting, and from \cite{CarusoWiestGeneric2} and \cite{WiestLox} in the large-ball setting.

The paper is organized as follows: in Section~\ref{S:Reminders and first results} we recall some basic facts about Garside groups, Artin-Tits groups, and the additional length graph. In Section~\ref{S:Lox} we construct the element $x_G$ and prove Theorem~\ref{T:InfDiam}. Section~\ref{S:Acylindrical} is devoted to the proof of Theorems~\ref{T:WPD} and~\ref{T:Acylindrical}, and Section~\ref{S:Generic} contains the results on generic elements of~$G$.
%

%

\section{Reminders}\label{S:Reminders and first results}

In this section we recall the concepts and results needed in the sequel. Except Section~\ref{S:CAL}, where we recall from~\cite{CalvezWiest} the construction of the additional length graph and discuss its diameter in the case of an irreducible Artin-Tits group of spherical type, all results are well-known and may be skipped by the experts.

\subsection{Garside theory}\label{S:Garside}
The reader is referred to~\cite[Section 2.1]{GebhardtGM} for a particularly simple introduction to Garside groups containing all what we need in this paper;
other more advanced references are~\cite{DehornoyParis},\cite{DehornoyGarside} and the book~\cite{DDGKM}.
A (\emph{finite-type}) \emph{Garside structure} for a group $G$ is the data of a submonoid $P$ of $G$ (the monoid of \emph{positive elements}) such that $P\cap P^{-1}=\{1\}$, together with a special element $\Delta$ of $P$ (the \emph{Garside element}) which enjoy the following properties: 
\begin{itemize}
\item The partial order $\preccurlyeq$ defined on $G$ by $x\preccurlyeq y$ if and only if $x^{-1}y\in P$ is a lattice order (invariant under left multiplication) called \emph{prefix order}: every two elements $x,y$ of $G$ admit a unique least common multiple (lcm) $x\vee y$ and a unique greatest common divisor (gcd) $x\wedge y$. 
\item The set of \emph{simple elements}: $\{s\in G: 1\preccurlyeq s \preccurlyeq \Delta\}$ (is \emph{finite} and) generates $G$. 
\item Conjugation by $\Delta$ preserves $P$ and hence $\preccurlyeq$. 
\item $P$ is noetherian: every non trivial element $x\in P$ satisfies: 
$$\sup\{k\in \mathbb N,\ x=a_1\ldots a_k,\  a_i\in P-\{1\}\}<\infty.$$
\end{itemize}
In particular, $P$ is atomic: there are elements $a\in P$ such that $a=uv$ and $u,v\in P$ imply $u=1$ or $v=1$; these elements are called \emph{atoms}. The set of atoms (is \emph{finite} and) generates $G$.

A group equipped with a (finite-type) Garside structure is termed \emph{Garside group (of finite type)}. In the sequel, we consider a Garside group $G$ with finite-type Garside structure $(G,P,\Delta)$. We denote by $\tau$ the conjugation by $\Delta$: for every $x\in G$, $\tau(x)=\Delta^{-1} x \Delta$. Analogue to the prefix order is the \emph{suffix order} $\succcurlyeq$: $x\succcurlyeq y\Leftrightarrow xy^{-1}\in P$; this is also a lattice order (invariant under right multiplication) and the associated gcds and lcms are denoted respectively by $\wedgeR$ and $\veeR$. Simple elements are simultaneously the positive prefixes and suffixes of $\Delta$; to each non-trivial simple element $s$ are associated its \emph{right complement} $\partial s$ and its \emph{left complement} $\partial^{-1}s$: both are simple elements defined respectively by $\partial s=s^{-1}\Delta$ and $\partial ^{-1}s=\Delta s^{-1}$. 
Conjugation by $\Delta$ induces a permutation of the set of simple elements, which is finite and generates $G$; it follows that $\tau$ is a finite order automorphism of $G$. 
We denote by~$o(G)$ the order of~$\tau$; thus $\Omega=\Delta^{o(G)}$ is the least central power of~$\Delta$. 

An ordered pair of simple elements $(s_1,s_2)$ is said to be \emph{left-weighted} if $\partial s_1\wedge s_2=1$ and \emph{right-weighted} if $s_1\wedgeR \partial^{-1}s_2=1$. Using this definition, one can define the left and right normal forms for any element $x\in G$. Let $x\in G$. The \emph{supremum} of $x$ is $\sup(x)=\min\{k\in \Z,\ x\preccurlyeq \Delta^k\}$, the \emph{infimum} of $x$ is $\inf(x)=\max\{k\in Z,\ \Delta^k\preccurlyeq x\}$ and the \emph{canonical length} of $x$ is $\ell(x)=\sup(x)-\inf(x)$. 
Elements with zero canonical length are exactly the powers of the Garside element $\Delta$. 
Each $x\in G$ with $\ell(x)\geqslant 1$ can be uniquely written $x=\Delta^{\inf(x)} s_1\ldots s_{\ell(x)}$, where $s_i$ are non-trivial simple elements, $s_1\neq \Delta$, so that each pair of consecutive factors is left-weighted. This is the \emph{left normal form} of $x$. The \emph{right normal form} is defined similarly: it is the unique way of writing $x=s'_{\ell(x)}\ldots s'_{1}\Delta^{\inf(x)}$ with the $s'_i$ are non-trivial simple elements, $s'_1\neq \Delta$, so that each pair of consecutive factors is right-weighted. 

For a given element $x\in G$ of canonical length at least 1 and left normal form $\Delta^px_1\ldots x_r$, we define $\iota(x)=\tau^{-p}(x_1)$ and $\varphi(x)=x_r$. These simple elements are called the \emph{initial} and \emph{final factor} of $x$, respectively.  
An element $x$ of $G$ (with $\ell(x)\geqslant 1$) is termed \emph{(left-)rigid} if its left normal form is a word "cyclically in left normal form", that is, if the pair $(\varphi(x),\iota(x))$ is left-weighted.

We denote by $G^{0}$ the subset of all (necessarily positive) elements with infimum 0. The notion of left and right complement can be extended to all elements of $G^0$: let $x\in G^0$, then the \emph{left} (respectively \emph{right}) \emph{complement} of $x$ is given by $\partial^{-1}x=\Delta^{\sup(x)}x^{-1}$ ($\partial x=x^{-1}\Delta^{\sup(x)}$, respectively). We say that an ordered pair of elements $(x,y)$ of $G^0$ is \emph{in left normal form} if the left normal form of the product $xy$ is the concatenation of the left normal form of $x$ followed by the left normal form of $y$ (in particular, $xy\in G^0$ and $\sup(xy)=\sup(x)+\sup(y)$).

Another useful notion of normal form is the so-called \emph{mixed canonical form}, introduced by Thurston~\cite{Thurston}. First, for every $x\in G$, there exists a unique irreducible fractional decomposition of $x$: a pair $(p,n)$ of elements of $P$ satisfying $n\wedge p=1$ such that $x=n^{-1}p$. 
(Notice that, if both $n$ and $p$ are non-trivial, then both belong to~$G^0$.)
Thurston's normal form is then obtained by further decomposing $n$ and $p$ into their respective left normal forms. 

To conclude this reminder section, we state the following easy but useful lemma, which will be needed in the sequel. 

\begin{lemma}\label{Lemma:Garside}
Let $u,v\in G^0$ and let $r=\sup(u)\geqslant 1$. For all $0\leqslant k\leqslant r$, denote by~$\phi_k$ the product of the rightmost $k$ factors of the right normal form of~$u$ ($\phi_0=1$) and by~$\iota_k$ the product of the leftmost $k$ factors of the right normal form of~$u$ ($\iota_0=1$). 
Then $\inf(uv)=\max\{k\in\{0,\ldots, r\},\ \partial \phi_k\preccurlyeq v\}$. If $\inf(uv)=k$, then $\iota_{r-k}\preccurlyeq uv\Delta^{-k}$.
\end{lemma}
\begin{proof} First recall that for any $x\in G$, $k\in\mathbb Z$, we have $\Delta^k\preccurlyeq x$ if and only if $x\succcurlyeq \Delta^k$. 
Note also that $\inf(uv)\geqslant \inf(u)+\inf(v)=0$ and $\inf(uv)\leqslant r$ because otherwise $\Delta^{r+1}=u\partial u \Delta\preccurlyeq uv$, which would imply that $\partial u\Delta\preccurlyeq v$, in contradiction with $\inf(v)=0$.  
Thus $\inf(uv)\in \{0,\ldots, r\}$.  

Let $k_0=\inf(uv)$ and $k'_0= \max\{k\in\{0,\ldots, r\},\ \partial \phi_k\preccurlyeq v\}$. 
We have on the one hand $uv\succcurlyeq \Delta^{k_0}$, that is $\Delta^{k_0}=\Delta^{k_0}\wedgeR uv$. But it is known \cite[Proposition 2.1]{michel} that the latter is exactly $\Delta^{k_0}\wedgeR \left((\Delta^{k_0}\wedgeR u)v\right)$, that is $\Delta^{k_0}=\Delta^{k_0}\wedgeR\phi_{k_0}v$. From this, it follows that $\Delta^{k_0}\preccurlyeq \phi_{k_0}v$, whence $\partial \phi_{k_0}\preccurlyeq v$. This means that $k_0\leqslant k'_0$. 

On the other hand, there is a positive $p$ such that $v=\partial \phi_{k'_0}p$; from this it follows that 
$$uv=(\iota_{r-k'_0}\phi_{k'_0})(\partial\phi_{k'_0}p)=\Delta^{k'_0}\tau^{k'_0}(\iota_{r-k'_0}) p,$$
whence $k'_0\leqslant k_0$. This achieves the first claim of the lemma. 

The remaining statement now follows easily. We saw that $\partial \phi_{k_0}\preccurlyeq v$, i.e.\ there is a positive~$q$ such that $v=\partial \phi_{k_0} q$. Then $uv\Delta^{-k_0}=\iota_{r-k_0}\phi_{k_0}\partial \phi_{k_0} q\Delta^{-k_0}=\iota_{r-k_0}\tau^{-k_0}(q)$, so that $\iota_{r-k_0}\preccurlyeq uv\Delta^{-k_0}$. 
\end{proof}

\subsection{Artin groups of spherical type and their classical Garside structure}\label{S:ArtinGroups}

A \emph{Coxeter matrix} is a square symmetric matrix $(m_{ij})_{1\leqslant i,j\leqslant n}$ 
of size $n\geqslant 2$ with diagonal coefficients equal to 1 and $m_{ij}\in \{2,3,\ldots, \infty\}$ for $i\neq j$; we will make the additional assumption that at least one of $m_{i,j}$  is at least 3.

This data can be encoded by a \emph{Coxeter diagram}; this is a labelled graph with $n$ vertices $v_1,\ldots, v_n$, where two vertices $v_i$ and $v_j$ are connected by an edge if $m_{i,j}\geqslant 3$ and edges are labeled by $m_{i,j}$ whenever $m_{i,j}\geqslant 4$. Our additional assumption says that this graph has at least one edge. 

To any Coxeter matrix $M=(m_{i,j})$ of size $n$ one can associate the group presentation: 
$$\left\langle s_1,\ldots,s_n | \underbrace{s_is_j\ldots}_\textrm{$m_{i,j}$ alternating terms}=  \underbrace{s_js_i\ldots}_\textrm{$m_{i,j}$ alternating terms},\ \ \forall \ 1\leqslant i\neq j \leqslant n\right\rangle.$$ 

A group admitting such a presentation is called an \emph{Artin-Tits group} (or simply \emph{Artin group}). For uniformity of the statements in the sequel, we wanted to exclude free abelian groups from this definition, this is the reason for our additional assumption on the definition of a Coxeter matrix. 

If $M$ is a Coxeter matrix, we denote by $A_M$ the Artin group defined by the presentation associated to $M$. 
If moreover we add to this presentation the relation $s_i^2=1$ for every generator $s_i$, this defines a \emph{Coxeter group}: the Coxeter group associated to $M$ in this way will  be denoted $W_M$. We say that the Artin group $A_M$ is of 
\emph{spherical type} if the Coxeter group $W_M$ is finite. When the Coxeter diagram associated to the Coxeter matrix $M$ is connected, 
the groups $A_M$ and $W_M$ are said to be \emph{irreducible}, otherwise \emph{reducible}.

Finite irreducible Coxeter groups (and hence irreducible Artin groups of spherical type) were completely classified by Coxeter in 1935~\cite{Coxeter}. This well-known result is contained in Figure~\ref{F:CoxeterClassification}.

\begin{figure}[hbt]
\begin{center}\includegraphics[scale=0.88]{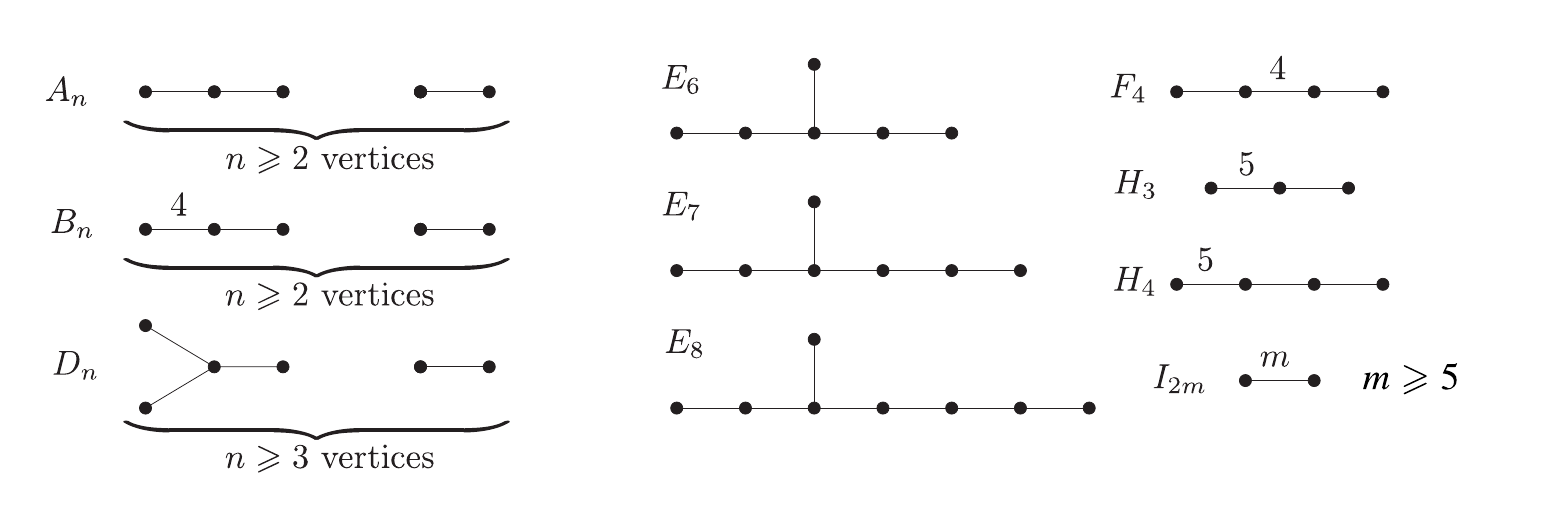}
\end{center}
\caption{If $M$ is a Coxeter matrix such that $A_M$ is irreducible and of spherical type, then the Coxeter diagram associated to $M$ is one of the graphs in this list.} 
\label{F:CoxeterClassification}
\end{figure}

It is known that each Artin group of spherical type admits two distinct Garside structures: the \emph{classical} structure was discovered by Brieskorn-Saito and Deligne \cite{BS,Del} as a generalization of Garside's work~\cite{Garside} on the braid group $B_n$ (that is, the Artin group of type $A_{n-1}$). On another hand, the \emph{dual} Garside structure was introduced more recently by Birman, Ko, Lee and Bessis~\cite{bkl, bessis}. Throughout this paper we are concerned with the former structure. As an aside, we notice that the free abelian group of rank $n\geqslant 1$ is also a Garside group, where the classical Garside structure just consists of the monoid $\N^n$ together with the Garside element $\Delta=(1,\ldots,1)$. 

Now, let $G$ be an \emph{irreducible} Artin group of spherical type; we describe briefly the classical Garside structure of $G$. 
A presentation for $G$ is provided by one of the graphs in Figure~\ref{F:CoxeterClassification}; say $\Gamma$. 
Let $V$ be the set of vertices of $\Gamma$; it corresponds bijectively with the set of generators in the defining presentation of $G$. 
%
The monoid of the positive elements in the classical Garside structure of $G$ is the monoid defined by the same presentation; its elements are those elements of $G$ which can be written as a product of positive powers of the given generators. These generators are atoms and we will use the same notation $V$ for the set of vertices of $\Gamma$ and for the set of atoms in $G$, and $v\in V$ will be, according to the context, a vertex of $\Gamma$ or an atom in $G$. 
The Garside element $\Delta$ is the least common multiple (with respect to both the prefix and the suffix order) of all the atoms \cite[Section 5]{BS}. A positive element of $G$ is a simple element if and only if all possible words in $V$ representing it are square-free \cite[5.4]{BS}. Also, to each subset $V'$ of $V$ is associated a "partial Garside element": $\Delta_{V'}$ is the least common multiple (both on the left and on the right) of the atoms in $V'$ -- note that $\Delta_{V'}$ is a simple element. 
The set $\mathcal S$ of simple elements is in bijective correspondence with the corresponding (finite) Coxeter group \cite[5.6]{BS}. The least central power of $\Delta$ is $o(G)=2$ if $\Gamma$ is one of $A_n$ $(n\geqslant 2)$, $D_{2k+1}$, $E_6$ or $I_{2(2q+1)}$ and $o(G)=1$ otherwise; in any case, the cyclic subgroup $\langle\Omega\rangle=\langle\Delta^{o(G)}\rangle$ is \emph{exactly} the center $Z(G)$ \cite[7.2]{BS}.

The conditions of left and right-weightedness for simple elements can be conveniently rewritten using the notions of \emph{starting} and \emph{finishing sets}. Let $s$ be a simple element. The \emph{starting set} $S(s)$ of $s$ is the set of atoms which are prefixes of $s$; the \emph{finishing set} $F(s)$ of $s$ is the set of atoms which are suffixes of~$s$. With this notation, the pair of simple elements $(s_1,s_2)$ is left-weighted if and only if $S(s_2)\subset F(s_1)$; it is right-weighted if and only if $F(s_1)\subset S(s_2)$. Finally, we recall that $G$ is equipped with the so-called \emph{reverse antiautomorphism}, which sends an element $x$ of $G$ represented by a word $w$ in $V\cup V^{-1}$ to the element $rev(x)$  represented by the word $w$ written backwards.

\subsection{The additional length graph}\label{S:CAL}

We now recall from~\cite{CalvezWiest} the construction of the \emph{additional length graph}. Throughout, $(G,P,\Delta)$ is an arbitrary finite type Garside structure for  the Garside group $G$. The main technical definition is that of an absorbable element:
\begin{definition}
We say that an element $y\in G$ is \emph{absorbable} (in an element $x\in G$) if the two following conditions are satisfied:
\begin{itemize}
\item $\inf(y)=0$ or $\sup(y)=0$,
\item there exists $x\in G$ such that $\inf(xy)=\inf(x)$ and $\sup(xy)=\sup(x)$. 
\end{itemize}
\end{definition} 

As an easy consequence of the definition, notice \cite[Lemma 2.3]{CalvezWiest} that an element $y$ of~$G$ is absorbable if and only if its inverse is absorbable (as well as its image under~$\tau$). Also, we recall \cite[Lemma 2.4]{CalvezWiest} that if $y$ is a positive absorbable element of $G$, then for any decomposition $y=uvw$ with $u,v,w\in P$, $v$ is absorbable. 

\begin{example}\cite[Example 2.6 (4)]{CalvezWiest}\label{example}
Suppose that $a$ is an atom of $G$. Then the right and the left complements of $a$ are not absorbable.  
\end{example}

\begin{definition}
The \emph{additional length graph} associated to the Garside structure $(G,P,\Delta)$ is the graph defined by the following data: 
\begin{itemize}
\item Vertices are in correspondence with the left cosets $g\Delta^{\Z}$. Each vertex $v$ has a unique \emph{distinguished representative} $\underline v\in G^0$. 
\item Two vertices $v,w$ are connected by an edge if and only if one of the following happens: 
\begin{itemize}
\item there is a simple element $s$ distinct from 1 and $\Delta$ such that $\underline v s$ belongs to the coset $w$ (equivalently there is a simple element $s'$ distinct from 1 and $\Delta$ such that $\underline w s'$ belongs to the coset $v$), or
\item there is an absorbable element $y\in G$ such that $\underline v y$ belongs to the coset $w$ (equivalently there exists an absorbable element $y'$ such that $\underline w y'$ belongs to the coset $v$). 
\end{itemize}
\end{itemize}
The additional length graph associated to $(G,P,\Delta)$ is denoted $\mathcal C_{AL}(G,P,\Delta)$; for short we also denote sometimes $\mathcal C_{AL}(G)$  or even $\mathcal C_{AL}$. 
\end{definition} 
As usual, the graph $C_{AL}$ is equipped with a metric structure, declaring that each edge has length 1. We call this metric the \emph{additional length metric} and we denote $d_{AL}(u,v)$ the distance between two vertices $u$ and $v$.  The group $G$ acts isometrically on the left on $\mathcal C_{AL}(G)$ by $g\cdot v=(g\underline v)\Delta^{\mathbb Z}$. Observe that $\Delta$ may not act trivially but $\Delta^{o(G)}$ does; therefore the quotient $G\!_{\Delta}=G/\left\langle\Delta^{o(G)}\right\rangle$ acts isometrically on $\mathcal C_{AL}(G)$. 

In~\cite{CalvezWiest}, we defined a family of \emph{preferred paths} in $\mathcal C_{AL}$ that we recall now. Let $u,v$ be vertices of $\mathcal C_{AL}$; consider the left normal form $s_1\ldots s_r$ of the distinguished representative of the vertex $\underline 
u^{-1}\cdot v=(\underline u^{-1}\underline v)\Delta^{\Z}$. The \emph{preferred path} 
$A(u,v)$ is the path of length $r$ starting at $u$ whose edges are successively labelled $s_1,\ldots, s_r$. 
The following properties were shown in~\cite{CalvezWiest} (Lemmas 2.11 and 2.10, respectively): 
\begin{itemize} 
\item[(PP1)] 
Preferred paths are symmetric: for any vertices $u,v$ of $\mathcal C_{AL}$, $A(v,u)$ meets (in the opposite order) the same vertices as $A(u,v)$.
\item[(PP2)]  The preferred path $A(u,v)$ between two vertices $u$ and $v$ is the concatenation of the paths 
$A(u,(\underline u\wedge \underline v)\Delta^{\Z})$ and $A((\underline u\wedge \underline v)\Delta^{\Z},v)$. 
\end{itemize}

Moreover, we state the following straightforward fact, which will be useful in the sequel. 
\begin{lemma}\label{L:Equivariant}
Let $u,v$ be two vertices of $\mathcal C_{AL}$. Let $g\in G$. If $u=u_0,u_1, \ldots, u_m=v$ is the sequence of vertices along the preferred path $A(u,v)$, then the preferred path $A(gu,gv)$ 
meets the vertices $gu=gu_0, gu_1,\ldots, gu_m=gv$. Moreover, if $s_1,\ldots, s_m$ is the sequence of edge labels along the path $A(u,v)$, then the sequence of edge labels along $A(gu,gv)$ is either 
$s_1,\ldots, s_m$ or $\tau(s_1),\ldots, \tau(s_m)$.
\end{lemma}

Finally, the following result was the main theorem of~\cite{CalvezWiest}: 

\begin{theorem}\label{T:QuasiGeodesics}
For any Garside group $(G,P,\Delta)$, the additional length graph $\mathcal C_{AL}(G,P,\Delta)$ is 60-hyperbolic. 
The preferred paths are uniform unparameterized quasi-geodesics in~$\mathcal C_{AL}$: for all vertices $u,v$ of $\mathcal C_{AL}$, the Hausdorff distance between $A(u,v)$ and any geodesic relating $u$ and $v$ is bounded above by 39. 
\end{theorem}

Our next goal is to study the diameter of the additional length graph associated to the classical Garside structure of an Artin group of spherical type. Before proceeding, we state a small technical remark about a kind of inverse for the vertices of the additional length graph. 

\begin{lemma}\label{L:Inverse}
Let $g\in G$. Consider the vertices $v=g\Delta^{\mathbb Z}$ and $v'=g^{-1}\Delta^{\Z}$ of $\mathcal C_{AL}(G)$. 
We have $\underline v'=\tau^{\inf(g)}(\partial \underline v)$. 
\end{lemma}
\begin{proof}
The distinguished representative of $v$ is $\underline v=g\Delta^{-\inf(g)}$. We now observe from the latter equality that 
$$g^{-1}=\Delta^{-\inf(g)}\underline v^{-1}=\Delta^{-\inf(g)} \partial \underline v\Delta^{-\sup(\underline v)}=
\tau^{\inf(g)}(\partial \underline v)\Delta^{-\sup(\underline v)-\inf(g)},$$ which shows that the element $\tau^{\inf(g)}(\partial \underline v)$ of $G^0$ is the distinguished representative of~$v'$.~\end{proof}

To conclude this section, we proceed to show the "only if" part of the statement of Theorem~\ref{T:InfDiam}:

\begin{proposition}\label{P:OnlyIf}
Suppose that $G$ is a reducible (nonabelian) Artin group of spherical type, endowed with its classical Garside structure. Then the diameter of $\mathcal C_{AL}(G)$ is finite.  
\end{proposition}

As a warm-up before giving the proof, we deal with the case of the additional length graph associated to the classical Garside structure of a free abelian group. 

\begin{proposition}
Let $n\geqslant 1$. The additional length graph $\mathcal C_{AL}(\mathbb Z^n)$ has finite diameter, except if $n=2$.
\end{proposition}
\begin{proof}
For $n=1$, the additional length graph has only one vertex, by definition. For $n=2$, the left cosets of $\Delta$ form a cyclic group (the quotient of $\mathbb Z^2$ by its diagonal). Moreover, no non-trivial element is absorbable so that the additional length graph is just an infinite path graph. Finally, the case $n\geqslant 3$ is dealt with in Example 2.6~(i) in~\cite{CalvezWiest}: any power of an atom is absorbable, so that the diameter of $\mathcal C_{AL}$ is at most~$n$. 
\end{proof}

\begin{proof}[Proof of Proposition~\ref{P:OnlyIf}]
Suppose that $G=G_1\times \cdots\times G_m$ is reducible with $m\geqslant 3$ components. For $i=1,\ldots,m$ we denote by $\Delta_i$ the Garside element of $G_i$, so that the Garside element of $G$ is $(\Delta_1,\ldots,\Delta_m)$. For $x\in G$ and $i\in \{1,\ldots,m\}$, we write $x_i$ for the $i$th coordinate of $x$. 
Let $i\in \{1,\ldots,m\}$; let $y \in G$ such that $y_i$ is a \emph{positive} element of $G_i$ and $y_j$ is trivial for $j\neq i$. Choose $j\neq i$. Let $z\in G$ such that $z_j=\Delta_j^{\sup(y_i)}$ and $z_k$ is trivial for $k\neq j$. Then we see that $y$ is absorbable in $z$. 
Using the mixed canonical form in each factor group, we see that each element of $G$ with all its coordinates trivial but one can be decomposed as a product of at most 2 absorbable elements. Finally, at most $m$ elements with only one non-trivial coordinate are needed to write down any element of $G$. Hence $\mathcal C_{AL}(G)$ has finite diameter, at most~$2m$.

Now suppose that $G=G_1\times G_2$ is reducible with two components, at least one of which is not cyclic, say $G_1$ (so that $G_1$ has at least 2 atoms). Let $\Delta_1$ and $\Delta_2$ be the respective Garside elements; let $a_1$ be an atom in $G_1$ and $a_2$ be an atom in $G_2$. Every element of the form $(1,g_2)$ with $g_2$ \emph{positive} in $G_2$ is absorbable by $(a_1^{\sup(g_2)},1)$ -- this is true because $\inf(a_1^{\sup(g_2)})=0$. Using fractional decomposition in $G_2$, any element whose first coordinate is trivial can be decomposed as a product of at most two absorbable elements. 
We now claim that every element whose second coordinate is trivial can be decomposed as a product of at most \emph{3} absorbable elements. First, if $g_1\in G^0_1$, $(g_1,1)$ is absorbable by $(1,a_2^{\sup(g_1)})$. Second, as $G_1$ has at least 2 atoms, for any integer $k$, the first case shows that $(\Delta^{k}_1,1)=(a^k_1,1)(a^{-k}_1\Delta^k_1,1)$ is a product of two absorbable elements. Third, our claim about a general  $(g_1,1)$ is shown using the mixed canonical form and the two latter observations. Finally, we have shown that $\text{diam}(\mathcal C_{AL}(G))\leqslant 5$. 
\end{proof}

\section{Loxodromic elements}\label{S:Lox}

In this section, we shall complete the proof of Theorem~\ref{T:InfDiam}. Throughout, $G$ will be an \emph{irreducible} Artin group of spherical type, equipped with its classical Garside structure. We need to show that the additional length graph of $G$ has infinite diameter. A key-point is to construct suitable elements of $G$, whose action on the graph is loxodromic. 
This is accomplished by showing: 

\begin{proposition}\label{P:Lox}
There exist atoms $a,b\in G$ and an element $x_G\in G^0$ (i.e. $\inf(x_G)=0$)
satisfying the following properties: 
\begin{itemize}
\item[(i)] the left normal form and the right normal form of $x_G$ are the same (we shall just call it ``the normal form'' of~$x_G$),
\item[(ii)] $\iota(x_G)=\phi(x_G)=a$, and the first and last factors of the normal form of $x_G$  are the only ones consisting of a single atom, 
\item[(iii)] the normal form of $x_G$ has two consecutive factors which are respectively $b^{-1}\Delta$ and $\tau(b^{-1})\Delta$,
\item[(iv)] $\sup(x_G)$ is a multiple of $o(G)$; in other words, $\tau^{\sup(x_G)}=id_G$, and $\Delta^{\sup(x_G)}$ is central.
\end{itemize}
\end{proposition}

\begin{remark}
In our case, the case of irreducible Artin-Tits groups of spherical type, condition~(iv) is satisfied if the length of~$x_G$ is even. We will actually be able to choose the length of~$x_G$ to be even and at most 12 in all our cases.
\end{remark}

\begin{proof}[Proof of Proposition~\ref{P:Lox}]
The argument follows the proof by Gebhardt and Tawn \cite[Section 5]{GebhardtTawn} of the fact that the language of left normal forms in the positive monoid of $G$ is \emph{essentially 5-transitive}. This means in particular that for each pair of simple elements $s,t\notin \{1,\Delta\}$ in $G$, there is $x\in G^0$ (with canonical length at most $6$) such that $\iota(x)=s$ and $\phi(x)=t$. 
The proof of this fact in~\cite{GebhardtTawn} is given in Propositions 51 to 60, which successively deal with each specific case of Coxeter's classification. An exhaustive checking shows that if moreover $s=a$ is an atom and $t=b^{-1}\Delta$ is the complement of an atom, then $x$ can be chosen so that its left and right normal forms are the same and so that the number of atoms in each successive factor is strictly increasing, this implies in particular that no other factor than the first one is an atom.

As $(b^{-1}\Delta,\tau(b^{-1})\Delta)$ is both a left and right-weighted pair, the theorem will be proved if we can show that there is an element $x'\in G^0$ with the same left and right normal forms so that $\iota(x')=\tau(b^{-1})\Delta$ and $\phi(x')=a$. 
The reverse of the element $x$ above does the job. So the element $x_G=xx'$ enjoys all the conditions in the theorem and its canonical length is even, at most 12. 
\end{proof}

From now on, we fix atoms $a,b\in G$ and an element $x_G\in G^0$ satisfying the conclusions of Proposition~\ref{P:Lox}. We denote by $r$ its canonical length. 

\begin{lemma}\label{L:XRigid}
\begin{itemize} 
\item[(a)] $x_G$ is rigid. 
\item[(b)] $x_G$ and $\partial x_G$ are not absorbable. 
\item[(c)] The left and right normal forms of~$\partial x_G$ coincide, and $\partial x_G$ is also rigid. Moreover, for each $m\geqslant 1$, we have $\partial (x^m)=(\partial x)^m$ and we will write $\partial x^m$. 
\end{itemize}
\end{lemma}
\begin{proof}
(a) holds because the pair $a\cdot a $ is left-weighted.

(b) results respectively from Properties (ii) and (iii) of Proposition~\ref{P:Lox}, combined with Example~\ref{example} and \cite[Lemma 2.4]{CalvezWiest} 

For (c), in order to prove the rigidity of~$\partial x_G$, let $y_1\ldots y_r$ be the normal form of~$\partial x_G$. Then $y_1=x_r^{-1}\Delta=a^{-1}\Delta$ and $y_r=\tau^{r-1}(x_1^{-1})\Delta=\tau^{r-1}(a^{-1})\Delta = \Delta \tau^r(a^{-1}) = \Delta a^{-1}$. We see that the pair $(y_r, y_1)$ is left and right-weighted. 
\end{proof}

\begin{lemma}\label{L:InitialFinalFactors}
Let $z\in G^0$.
\begin{itemize} 
\item[(a)] Suppose that $k\geqslant 2$; if $x_G^k\preccurlyeq z$, then the left normal form of $z$ starts with $k-1$ copies of $x_G$.
\item[(b)] Suppose that $k\geqslant 1$; if $\partial x_G^k\preccurlyeq z$, then the left normal form of $z$ starts with $k$ copies of $\partial x_G$, 
\end{itemize}
\end{lemma}

\begin{proof} (a) is the statement of~\cite[Proposition 3.11]{CalvezWiest}. 

(b) We may assume that $z$ is not itself a power of $\partial x_G$ otherwise the result is trivial. Thus there is a positive non-trivial element $A$ such that $z=\partial x_G^k\,  A$ and we claim that this is the left normal form. Indeed, $\phi(\partial x_G^k)$ is the complement of an atom, so as $\inf(z)=0$, it must be that $\phi(\partial x_G^k)\cdot \iota (A)$ is left-weighted. 
\end{proof}

For simplicity, for every integer $k\in \Z$, we will denote by $X^k$ the vertex $x_G^k\Delta^{\Z}$ of $\mathcal C_{AL}(G)$. Notice that $\underline{X^k}=x_G^k$ if $k\geqslant 0$ and $\underline{X^k}=\partial x_G^{-k}$ if $k<0$;
also, for any $k\in \Z$, and any vertex $v$ in $\mathcal C_{AL}$, the preferred path $A(X^k,v)$ corresponds to the left normal form of the distinguished representative of the vertex $x_G^{-k}\cdot v$.

We consider the set $\{X^k,k\in \Z\}$; we think of it as a bi-infinite preferred path or equivalently as a path subgraph of $\mathcal C_{AL}$ made of the preferred paths (of length $r$) joining consecutive vertices. Our main tool will be a sort of projection from $\mathcal C_{AL}$ to this set. Before defining it we need to show: 

\begin{lemma}\label{L:Well-DefinedProjection}
Let $v$ be a vertex of $\mathcal C_{AL}$. Suppose that the non-negative integer $m$ satisfies $\partial x_G^m\preccurlyeq \underline v$ and $\partial x_G^{m+1}\not\preccurlyeq \underline v$.
Then for $j\geqslant 1$, the element $x_G^{j-1}$ is a prefix of the distinguished representative of $x_G^{m+j}\cdot v$.
\end{lemma}
\begin{proof}
By hypothesis, $\underline v=\partial x_G^m A$, where $A$ is a positive element of~$G$ such that $\partial x_G\not\preccurlyeq A$, and also $x_G\not\preccurlyeq A$ (because $\inf(\underline v)=0$). Let $j\geqslant 1$ and consider the vertex $x_G^{m+j}\cdot v$. It is represented by $x_G^{j+m}\partial x_G^m\, A=x_G^{j}\Delta^{rm}A=x_G^{j}A\Delta^{rm}$ (as $r$ is a multiple of~$o(G)$ and $\Delta^{o(G)}$ is central) and its distinguished representative is $\underline{x_G^{m+j}\cdot v}=x_G^jA\Delta^{-\inf(x_G^jA)}$. Now, we use Lemma~\ref{Lemma:Garside} (keeping in mind that the right and left normal form of $x_G^j$ are the same). 
As $\partial x_G$ is not a prefix of $A$, the lemma says that $\inf(x_G^jA)\leqslant r-1$ and also that 
$x_G^{j-1}\preccurlyeq x_G^jA\Delta^{-\inf(x_G^jA)}$, as claimed. 
\end{proof}

It follows in particular that for any vertex $v$, the set $\{k\in\Z,\ x_G\not\preccurlyeq \underline{x_G^k\cdot v}\}$ is bounded above. This allows us to define the desired projection. 

\begin{definition}\label{D:ProjToAxis}
Let $v$ be a vertex of $\mathcal C_{AL}$. We define the integer 
$$\lambda(v)=-\max\{k\in\mathbb Z, \ x_G\not\preccurlyeq \underline{x_G^{k}\cdot v}\}.$$ 
The projection is then given by $\pi(v)=X^{\lambda(v)}$. 
\end{definition}

\begin{example}
\label{Ex:Proj}
\begin{itemize}
\item [(a)] For $v=X^n=x_G^n\Delta^\mathbb Z$ (avec $n\in\mathbb Z$) we have $\lambda(v)=n$ and $\pi(v)=v$. 

\item [(b)] If $A\in G^0$ such that $\inf(x_G\, A)=0$ and such that $x_G\not\preccurlyeq A$ then for the vertex $v=x_G^n\, A\Delta^\mathbb Z$ (with $n\geqslant 1$) we have $\pi(v)=X^n$. 

\item [(c)] If $\lambda(v)\geqslant 1$, then $x_G^{\lambda(v)}\preccurlyeq \underline v$. If $\lambda(v)\leqslant -2$, then $\partial x_G^{-\lambda(v)-1}\preccurlyeq \underline v$. These two statements are not obvious -- indeed, they are special cases ($j=\lambda(v)$ and $j=-\lambda(v)$, respectively) of the following lemma.
\end{itemize}
\end{example}

\begin{lemma}\label{L:Partial}
Let $v$ be a vertex of $\mathcal C_{AL}$ and $\lambda=\lambda(v)$. 	
\begin{itemize}
\item [(a)] 
For every integer $j\geqslant 1$, the element~$x_G^j$ is a prefix of the distinguished representative of $x_G^{-\lambda+j}\cdot v$. 
\item[(b)] For every integer $j\geqslant 2$, the element~$\partial x_G^{j-1}$ is a prefix of the distinguished representative of $x_G^{-\lambda-j}\cdot v$.
\end{itemize}
\end{lemma}

\begin{proof}
(a) By definition of the projection, the distinguished representative of $x_G^{-\lambda+1}\cdot v$ can be written $x_GA$ for a positive $A$; from Lemma~\ref{Lemma:Garside}, it follows that $\inf(x_G^jA)=0$ for all $j\geqslant 1$, and hence the claim. 

(b) First, $x_G^{-j-\lambda}\cdot v= \partial x_G^j\Delta^{-rj}x_G^{-\lambda}\cdot v=\partial x_G^jx_G^{-\lambda}\cdot v$ (as $o(G)$ divides~$r$, so that $\Delta^r$ is central). 
Observe that $\partial x_G^j=(\partial x_G)^j$ has the same right and left normal forms, and that $\partial (\partial x_G)=x_G$, which is not a prefix of $\underline{x_G^{-\lambda}\cdot v}$, by definition of the projection. Lemma~\ref{Lemma:Garside} applies and says that $k=\inf(\partial x_G^{j}\underline{x_G^{-\lambda}\cdot v})\leqslant r-1$ and that $\partial x_G^{j-1}\preccurlyeq \partial x_G^{j}\underline{x_G^{-\lambda}\cdot v}\Delta^{-k}$, hence the claim.
$\Box$\end{proof}

The following is crucial and was already shown~\cite[Proposition 3.13]{CalvezWiest}; however the current proof, with the help of the projection defined above, becomes clearer than that in~\cite{CalvezWiest}. 

\begin{proposition}\label{P:Main}
Let $v_1,v_2$ be two vertices of the additional length graph; let $\lambda_1=\lambda(v_1)$, $\lambda_2=\lambda(v_2)$. 
Suppose that $\lambda_2\geqslant \lambda_1+3$. Then the preferred path $A(v_1,v_2)$ contains the subpath $A(X^{\lambda_1+1},X^{\lambda_2-1})$. 
\end{proposition}

\begin{proof} See Figure~\ref{Pic:MainProp}. 
Consider first the path $\gamma_2=A(X^{\lambda_1}, v_2)$. It is given by  the left normal form of the distinguished representative of the vertex $x_G^{-\lambda_1}\cdot v_2$. 
By Lemma~\ref{L:Partial}(a) (with $j=\lambda_2-\lambda_1$
and $\lambda=\lambda_2$), we know that $x_G^{\lambda_2-\lambda_1}\preccurlyeq \underline{x_G^{-\lambda_1}\cdot v_2}$. According to Lemma~\ref{L:InitialFinalFactors}(a), this says that the path $A(X^{\lambda_1},v_2)$ has $A(X^{\lambda_1},X^{\lambda_2-1})$ as an initial segment. 

\begin{figure}[hbt]
\begin{center}\includegraphics[height=3cm]{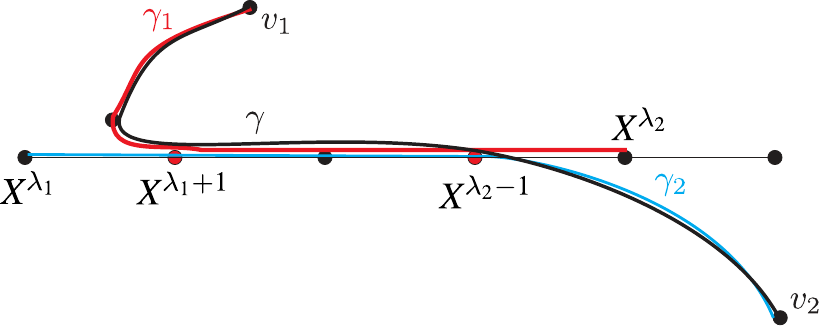}
\end{center}
\caption{The proof of Proposition~\ref{P:Main}.} \label{Pic:MainProp}
\end{figure}

Consider now the path $A(X^{\lambda_2},v_1)$. It is given by the left normal form of the distinguished representative of the vertex $x_G^{-\lambda_2}\cdot v_1$. By Lemma~\ref{L:Partial}(b) (with $j=\lambda_2-\lambda_1$), 
we have that $\partial x_G^{\lambda_2-\lambda_1-1}\preccurlyeq \underline{x_G^{-\lambda_2}\cdot v_1}$. By Lemma~\ref{L:InitialFinalFactors}(b), the left normal form of the latter starts with $\lambda_2-\lambda_1-1$ copies of $\partial x_G$. Because preferred paths are symmetric (PP1), this is equivalent to saying that $\gamma_1=A(v_1,X^{\lambda_2})$ has $A(X^{\lambda_1+1}, X^{\lambda_2})$ as a final segment. 

We see finally that $\gamma_1=A(v_1, X^{\lambda_2})$ and $\gamma_2$ coincide along the subpath $A(X^{\lambda_1+1}, X^{\lambda_2-1})$. Consider the path $\gamma$ formed by the subpath of $\gamma_1$ between $v_1$ and $X^{\lambda_2-1}$, followed by the subpath of $\gamma_2$  between $X^{\lambda_2-1}$ and $v_2$. Observe that $\gamma$ connects $v_1$ and $v_2$ and that the product of the labels of the successive edges along $\gamma$ gives a left normal form. This says that $\gamma = A(v_1, v_2)$. Since $\gamma$ has the claimed subpath, the proposition is shown. 
\end{proof}

We are now ready to finish the proof of Theorem~\ref{T:InfDiam}. We recall what is left to be done: throughout this section, we are dealing with the case of an~\emph{irreducible} Artin-Tits group of spherical type~$G$, and we have to prove that the graph~$\mathcal C_{AL}(G)$ is of infinite diameter. We will actually prove the stronger result that the element $x_G$ from Proposition~\ref{P:Lox} acts in a loxodromic fashion on the graph:

\begin{proposition}\label{P:XLox}
For any integer $N\in\Z$, 
$$d_{AL}\!\left(1,X^N\right)\geqslant \frac{|N|}{2}.$$
\end{proposition}
\begin{proof}
First we observe that it is sufficient to show the claim for $N>0$.
Suppose for a contradiction that there is an $N>0$ such that $k=d_{AL}(1,X^N)<\frac{N}{2}$. Let $v_0=1, v_1, \ldots v_k=X^N$ be the vertices along a path of length $k$ between $1$ and $X^N$. We have $\lambda(v_0)=0$ and $\lambda(v_k)=N$, so there must exist some $i$ between $0$ and $k-1$ such that $\lambda(v_{i+1})-\lambda(v_i)\geqslant 3$. By hypothesis, these two vertices are at distance 1 in $\mathcal C_{AL}$ so that they must be connected by a simple element or by an absorbable one. On the other hand, Proposition~\ref{P:Main} says that the preferred path between them contains the subpath $A(X^{\lambda(v_{i})+1},X^{\lambda(v_{i+1})-1})$. 
In view of Lemma \ref{L:XRigid}, as $x_G$ is neither absorbable nor simple (and neither is $\partial x_G$), this is a contradiction.
\end{proof}


\section{Acylindrical hyperbolicity}\label{S:Acylindrical}
Throughout this section, $G$ is an irreducible Artin group of spherical type. 
Recall that in this context, the least central power $o(G)$ of the Garside element $\Delta$ is 1 or 2 and 
that $\langle \Delta^{o(G)}\rangle$ is exactly the center of $G$. We consider the left-isometric action of the quotient $G\!_{\Delta}=G/Z(G)$ on $\mathcal C_{AL}(G)$. We fix  an element $x_G$ of $G^0$ as in Proposition~\ref{P:Lox}; we denote by $r$ its (even) canonical length. 
We know that the action of $x_G$ on $\mathcal C_{AL}(G)$ is loxodromic; of course the action of its class in $G\!_{\Delta}$ is loxodromic as well. By abuse of notation, we will also write $x_G$ for the class of~$x_G$ in~$G\!_{\Delta}$. 

Our goal now is to prove Theorems~\ref{T:WPD} and~\ref{T:Acylindrical}.
We first recall the definition of a Weakly Properly Discontinuous action from~\cite{BestvinaFujiwara02}: 

\begin{definition}
The loxodromic element $x_G$ of ${G}\!_{\Delta}$ is WPD if for every $v\in \mathcal C_{AL}$, and every $\kappa>0$, there exists $N>0$ such that the set $$\left\{\gamma \in {G}\!_{\Delta}\ |\  d_{AL}\!\left(v,\gamma\cdot v\right)\leqslant \kappa,\  d_{AL}\!\left(x_G^{N}\cdot v ,(\gamma x_G^{N})\cdot v\right)\leqslant \kappa\right\}$$ is finite.  
\end{definition}

The main step towards proving Theorem~\ref{T:WPD} is to prove it in the special case when $v$ is the identity vertex of ~$\mathcal C_{AL}$, i.e.\ to prove the following result:

\begin{proposition}\label{prop2}
For every $\kappa>0$, there exists $N>0$ such that the set 
$$\left\{\gamma \in {G}\!_{\Delta}\ | \ d_{AL}\!\left(1,\gamma\cdot 1\right)\leqslant \kappa,\  d_{AL}\!\left(x_G^{N}\cdot 1 ,(\gamma x_G^{N})\cdot 1\right)\leqslant \kappa\right\}$$ 
is finite.
\end{proposition}

\begin{proof}[Proof of Theorem~\ref{T:WPD}, assuming Proposition~\ref{prop2}]
Fix any vertex $v$ and any $\kappa>0$ and let $d=d_{AL}(1,v)$. Then by Proposition~\ref{prop2}, there exists $N>0$ so that the set 
$$\left\{\gamma \in {G}\!_{\Delta}\ |\  d_{AL}\!\left(1,\gamma\cdot 1\right)\leqslant \kappa+2d,\  d_{AL}\!\left(x_G^{N}\cdot 1 ,(\gamma x_G^{N})\cdot 1\right)\leqslant \kappa+2d\right\}$$ 
is finite. But it is a direct consequence of the triangle inequality that this finite set contains any element $\gamma\in G\!_{\Delta}$ satisfying $d_{AL}(v,\gamma\cdot v)\leqslant \kappa$ and ${d_{AL}(x_G^N\cdot v,(\gamma x_G^N)\cdot v)\leqslant \kappa}$. 
\end{proof}

\begin{proof}[Proof of Theorem~\ref{T:Acylindrical}]
This is a direct consequence of our Theorem~\ref{T:WPD}: it says that $G/Z(G)$ satisfies (AH3) of~\cite[Theorem 1.2]{osin}; but according to this theorem, this is equivalent to acylindrical hyperbolicity (AH2). 
\hfill $\Box$ \end{proof}

\begin{remark}
\cite[Theorem 1.2]{osin} should not be misinterpreted: the implication (AH3) $\Longrightarrow$ (AH2) does not mean that the action of $G/Z(G)$ \emph{on our space} $\mathcal C_{AL}$ is acylindrical.
It would be very interesting to know whether this is the case; we conjecture that the answer is positive. This is certainly suggested by the case of the Mapping Class Group acting on the curve graph~\cite{Bowditch08}.
\end{remark}

\begin{proof}[Proof of Proposition~\ref{prop2}]
Our strategy is to prove a much stronger result -- roughly speaking, we will prove that for large enough~$N$, the only such elements~$\gamma$ are small powers of $x_G$ (there is a slight complication due to the action of $\Delta$).

To be more precise, recall from Theorem~\ref{T:QuasiGeodesics} that the preferred paths in $\mathcal C_{AL}(G)$ are unparametrized quasigeodesics: in the sequel, we let $M$ be a positive constant such that for any vertices $v,w$ in $\mathcal C_{AL}(G)$, the Hausdorff distance between $A(v,w)$ and any geodesic between $v$ and $w$ is bounded above by $M$. Indeed, Theorem~\ref{T:QuasiGeodesics} says that one can take $M=39$. 

In Definition~\ref{D:ProjToAxis} we defined a projection $\pi$ from $C_{AL}$ to the axis $\{X^k \ | \ k\in\mathbb Z\}$ (where $X^k$ stands for the vertex $x_G^k\Delta^{\Z}$ of $\mathcal C_{AL}$), given by $\pi(v)=X^{\lambda(v)}$. We first prove that this projection is coarsely Lipschitz:


\begin{proposition}\label{P:Lipschitz}
Suppose that $v,w$ are vertices of $\mathcal C_{AL}$. Then 
$$|\lambda(w)-\lambda(v)|\leqslant 2(d_{AL}(v,w)+2M+1)$$
\end{proposition}

\begin{proof}
We may assume that $\lambda(w)-\lambda(v)\geqslant 3$. 
As seen in Proposition~\ref{P:Main}, it follows that $A(v,w)$ contains the subpath $A(X^{\lambda(v)+1},X^{\lambda(w)-1})$.  By definition of $M$, the triangle inequality implies that 
$$d_{AL}(X^{\lambda(v)+1},X^{\lambda(w)-1})\leqslant d_{AL}\!\left(v,w\right)+2M.$$
On the other hand, Proposition~\ref{P:XLox} says that
$$\lambda(w)-\lambda(v)-2=\lambda(w)-1-(\lambda(v)+1)\leqslant 2d_{AL}(X^{\lambda(v)+1},X^{\lambda(w)-1}).$$ 
Combining both inequalities, we obtain that 
$$\lambda(w)-\lambda(v)\leqslant 2(d_{AL}(v,w)+2M)+2,$$ 
as claimed in Proposition~\ref{P:Lipschitz}.
\end{proof}

Now, fix a $\kappa>0$. Let $\xi=\kappa+2M+1$. We also choose $N$ to be any integer with 
$N\geqslant 4\xi+3$. 
Consider an element $\gamma\in {G\!_{\Delta}}$ which satisfies 
$$\begin{cases}
d_{AL}(1,\gamma\cdot 1)\leqslant \kappa,\\
d_{AL}(x_G^{N}\cdot 1, (\gamma x_G^N)\cdot 1)\leqslant \kappa.\\
\end{cases}$$

For simplicity we shall use the following notation for vertices of~$\mathcal{C}_{AL}$: $u=\gamma\cdot 1$, $v=X^N=x_G^N\Delta^{\Z}$, and $w=\gamma\cdot v = \gamma x_G^{N}\Delta^{\Z}$.
The above conditions then read 
$$\begin{cases}
d_{AL}(1,u)\leqslant \kappa,\\
d_{AL}(v,w)\leqslant \kappa.\\
\end{cases}$$

\begin{lemma}\label{L:Subpath}
The path $A(u,w)$ contains the subpath $A(X^{\lambda(u)+1},X^{\lambda(w)-1})$.  
\end{lemma}

\begin{proof}
By Proposition~\ref{P:Main}, it suffices to prove that $\lambda(w)-\lambda(u)\geqslant 3$.

But by Proposition \ref{P:Lipschitz}, we have $|\lambda(w)-N|=|\lambda(w)-\lambda(v)|\leqslant 2(\kappa+2M+1)=2\xi$
and $|\lambda(u)|\leqslant 2(\kappa+2M+1)=2\xi$. 
Because of our choice of $N$, we have 
$$
\lambda(w)-\lambda(u)\geqslant N-4\xi \geqslant 4\xi+3-4\xi = 3
$$
proving Lemma~\ref{L:Subpath}. 
\end{proof}

Now, using Lemma \ref{L:Equivariant}, we see that $X^{\lambda(u)+1}=\gamma\cdot p$ and $X^{\lambda(w)-1}=\gamma\cdot q$, for some vertices $p,q$ along the preferred path $A(1,X^N)$. 
Moreover, the successive edges along the path 
$A(\gamma p,\gamma q)=A(X^{\lambda(u)+1},X^{\lambda(w)-1})$ 
wear the same labels (modulo conjugation by $\Delta$) as the edges along the path $A(p,q)$. 

Denoting by $x_1,\ldots, x_r$ the factors of the left normal form of $x_G$, the above amounts to saying that there is some 
$\epsilon\in \{0,1\}$ and an integer $j$ with $1\leqslant j\leqslant r$ such that 
the following two
$r\big(\lambda(w)-\lambda(u)-2\big)$-tuples of simple elements are equal: 
$$\big((x_1,\ldots, x_r),(x_1,\ldots, x_r),\ldots, (x_1,\ldots, x_r)\big)$$
and 
$$\Big(\big(\tau^{\epsilon}(x_j),\ldots, \tau^{\epsilon}(x_{j+r-1})\big),\big(\tau^{\epsilon}(x_j),\ldots,\tau^{\epsilon}(x_{j+r-1})\big),\ldots, \big(\tau^{\epsilon}(x_j),\ldots,\tau^{\epsilon}(x_{j+r-1})\big)\Big),$$
where the indices in the second tuple are taken modulo $r$. 
By construction of~$x_G$ this forces $j=1$: this is easily seen using (ii) of Proposition~\ref{P:Lox}. 
 
We deduce that $p$ and $q$ are vertices corresponding to some powers of $x_G$, say $p=X^m$ ($m\in\Z$) and $q=X^{m+\lambda(w)-\lambda(u)-2}$. But both of these are vertices along $A(1,X^N)$; therefore, 
$0\leqslant m$ and also $m+\lambda(w)-\lambda(u)-2\leqslant N$, which, by the inequalities in the proof of Lemma~\ref{L:Subpath}, 
implies $m\leqslant 4\xi+2$. 
In total we have $0\leqslant m\leqslant 4\xi+2$.

To conclude, we observe that the $\gamma$-action moves $X^m$ to $X^{\lambda(u)+1}$ (and we recall that $\gamma\in G\!_\Delta=G/\langle\Delta^{o(G)}\rangle$). Hence $\gamma$ is represented by $x_G^{\lambda(u)+1}\tau^{-j}(x_G^{-m})\Delta^j$ for some integer $j$; that is, if $m$ is fixed, $\gamma\in G\!_{\Delta}$ can run in a set with at most $o(G)$ elements. 
This finally implies that (if~$m$ is variable) $\gamma$ ranges over a set of 
$(4\xi+3) o(G)$ possibilities.

This completes the proof of Proposition \ref{prop2}. In summary, 
we have shown 
that for any $\kappa>0$, there exists an integer $N$, namely any $N$ satisfying 
$$N\geqslant 4\xi+3=4\kappa+8M+4+3=4\kappa+319$$
(recalling that $M=39$) such that the set 
$$\{\gamma\in G\!_{\Delta} \ | \ d_{AL}(1,\gamma\cdot 1)\leqslant \kappa, \ d_{AL}(x_G^N\cdot 1,(\gamma x_G^N)\cdot 1)\leqslant \kappa\}$$ 
has at most 
$(4\kappa+319)\cdot o(G)\leqslant 8\kappa+638$ elements.
\end{proof}

To conclude this section, recall that an element $g$ of $G\!_{\Delta}$ is called \emph{Morse} if every $(K,C)$-quasi-geodesic~$\sigma$ in the Cayley graph of~$G\!_\Delta$ with endpoints on~$\langle g\rangle$, is contained in a $\Lambda$-neighborhood of $\langle g\rangle$, 
where $\Lambda$ depends only on $K$,~$C$, and $g$ (but not on the choice of quasi-geodesic). 

\begin{corollary} Our special element $x_G$ is Morse. 
\end{corollary}
\begin{proof}
According to~\cite[Theorem 1.4]{osin}, the infinite-order element $x_G$ is contained in a virtually cyclic hyperbolically embedded subgroup of $G\!_{\Delta}$. According to Sisto \cite[Theorem 1]{sistohyp}, this implies that $x_G$ is Morse. 
\end{proof}

\section{Genericity in spherical type Artin groups}\label{S:Generic}

Throughout this section, let $G$ be an irreducible Artin group of spherical type, equipped with some finite generating set~$\mathcal S$.

Roughly speaking, the aim of this section is to prove that ``most'' elements of~$G$ (or ``generic'' elements, or ``random'' elements of~$G$) act loxodromically on the additional length graph $\mathcal C_{AL}(G)$.
In the case where $G$ is the braid group $G=B_n$, this is at least as strong as saying that generic braids are pseudo-Anosov (because non-pseudo-Anosov braids act elliptically on $\mathcal C_{AL}(B_n)$, see~\cite{CalvezWiest}).

To be more precise, there are several different ways to define what is meant by a ``random'' element. The two most common interpretations are
\begin{enumerate}
\item The random walk interpretation: perform a simple random walk of length~$N$ on the Cayley graph of~$G$, starting at the vertex representing the neutral element. Equivalently, choose a random word (with uniform probability) in the letters $\mathcal S\cup \mathcal S^{-1}\backslash \{1\}$ of length $N$. The ``random element'' of~$G$ is the element represented by this random word.
\item The interpretation using balls in the Cayley graph: consider the ball of radius~$N$ in the Cayley graph centered on the vertex representing the neutral element. Choose a vertex at random (with uniform probability) among the vertices in this ball.
\end{enumerate}
We are going to prove that, with either method for constructing random elements, the probability of obtaining an element which acts loxodromically on $\mathcal C_{AL}(G)$ tends to~1 exponentially quickly as $N$ tends to infinity. There is one caveat: if the ``balls in the Cayley graph'' framework is chosen, then we must use the generating set~$\mathcal S$ consisting of the simple elements in the classical Garside structure on~$G$.


\subsection{Elements obtained by a random walk are loxodromic and WPD}

In this subsection, the group~$G$ is equipped with any finite generating set~$\mathcal S$.

\begin{proposition}
The probability that an element of~$G$ obtained by a random walk of length~$N$ does not act loxodromically and WPD in $\mathcal C_{AL}(G)$ tends to zero exponentially fast as $N$ tends to infinity.
\end{proposition}

\begin{proof}
This is now an immediate consequence of the results of Sisto~\cite{SistoGeneric}.
Indeed, Sisto proved the following: if a non-elementary (i.e.\ not virtually cyclic) group $G$ acts on a $\delta$-hyperbolic space by isometries, with at least one element acting in a WPD manner, then the probability that a simple random walk of length~$N$ yields an element whose action is not WPD tends to zero exponentially quickly as $N$ tends to infinity.
\end{proof}


\subsection{Random elements from a ball in the Cayley graph are loxodromic}

Throughout this subsection, the group~$G$ is equipped with the generating set~$\mathcal S$ consisting of the simple elements in the classical Garside structure on~$G$.

\begin{proposition}\label{P:BallGeneric}
Consider the ball of radius~$N$ in the Cayley graph of~$G$, centered on the neutral element.
The proportion of vertices in this ball representing non-loxodromically acting elements of~$G$ tends to zero exponentially quickly as $N$ tends to infinity.
\end{proposition}

The proof is a simple adaptation of the proof by Caruso and Wiest (\cite[Theorem 5.1]{CarusoWiestGeneric2} that ``most'' vertices in a large ball in the Cayley graph of the braid group~$B_n$ represent pseudo-Anosov elements.
Thus in the current section we must assume that the reader is familiar with that proof.

The only new ingredient in the proof of Proposition~\ref{P:BallGeneric} is the following lemma, which is of interest in its own right:

\begin{lemma}\label{L:xSubwordLox}
Suppose $g$ is a rigid element of~$G$ whose normal form contains the subword $x_G^{390}$ (where $x_G$ is the loxodromically and WPD-acting element constructed in Section~\ref{S:Lox}). Then $g$ acts itself loxodromically on~$C_{AL}(G)$.
\end{lemma}

\begin{proof}[Proof of the Lemma]
The proof is based on \cite[Lemma 3.2]{WiestLox}. Let us recall the setup for this result. We have the group~$G$ acting on the $\delta$-hyperbolic space~$\mathcal C_{AL}$. The group is automatic, so the language of normal form words satisfies \emph{a fortiori} the ``automatic normal form hypothesis" of~\cite{WiestLox}. Also, the family of paths in~$\mathcal{C}_{AL}$ given by
$$
\{ P, x_1.P, x_1 x_2.P, x_1x_2x_3.P, \ldots \ | \ x_1x_2x_3\ldots \text{ \ a normal form word}\}
$$
forms a family of unparameterized quasi-geodesics in~$\mathcal C_{AL}$ (where $P$ denotes the base point $\Delta^{\mathbb Z}$ of~$\mathcal C_{AL}$). In particular, the ``geodesic words hypothesis'' of~\cite{WiestLox} is satisfied.
More precisely, normal form words stay no farther than distance $M=39$ from any geodesic connecting their endpoints (Theorem~\ref{T:QuasiGeodesics}).

Now~\cite[Lemma 3.2]{WiestLox} states that a rigid element cannot act parabolically, and if it acts elliptically, then its normal form cannot contain any subword representing an element whose action moves the base point by more than $5\cdot M=195$.

However, we know from our Proposition~\ref{P:XLox} that $x_G^{2\cdot 5\cdot 39}=x_G^{390}$ moves the base point by at least $5\cdot M$.
This implies that the rigid element $g$ whose normal form contains $x_G^{390}$ acts neither parabolically, nor elliptically.
\end{proof}

\begin{proof}[Proof of Proposition~\ref{P:BallGeneric}]
The proof follows very closely the proof of the genericity of pseudo-Anosov braids in~\cite{CarusoWiestGeneric2}.

The first observation is that $G$ contains \emph{blocking elements}, in the sense of~\cite[Definition 4.4]{CarusoWiestGeneric2}.
We need not concern ourselves with the precise definition of blocking elements right now;
it suffices to know that all elements with $\inf=0$, whose left and right normal forms coincide, with last factor consisting of a single atom, and with first factor of the form $a^{-1}\Delta$, where $a$ is an atom, are blocking elements.
Proposition~\ref{P:Lox} asserts the existence of such elements; typically the second half of our loxodromic element $x_G$ does the job -- in fact, for the classical Garside structure on an irreducible Artin-Tits group of spherical type, blocking elements can be chosen so as to have canonical length at most~6.

Once we know that blocking elements exist, we can proceed almost literally as in~\cite{CarusoWiestGeneric2}.
We denote
$$
G^{\eta,l} = \{ g\in G \ | \ \inf(g)=\eta,  \ell(g)=l \}
$$
and we look at the proportion of elements of~$G^{\eta, l}$ which admit a \emph{non-intrusive conjugation} (\cite[Definition 3.1]{CarusoWiestGeneric2}) to a rigid element. We prove an analogue of~\cite[Proposition~3.4]{CarusoWiestGeneric2}: there is a lower bound on this proportion which is independent of~$\eta$, and which tends to~$1$ exponentially quickly as $l$ tends to infinity.

On the other hand, we also look at the proportion of elements of~$G^{\eta, l}$ whose normal form contains in its middle fifth the word $x_G^{390}$. We prove that, similarly,  there is a lower bound on this proportion which is independent of~$\eta$, and which tends to~$1$ exponentially quickly as $l$ tends to infinity. (This proof is based on the techniques of~\cite{CarusoGeneric1}.)

Precisely as in~\cite{CarusoWiestGeneric2}, looking at the intersection of these two subsets of $G^{\eta, l}$, we deduce that the proportion of elements admitting a non-intrusive conjugation to a rigid, loxodromically acting element tends to 1 exponentially quickly as $l$~tends to infinity.

Finally, we partition the ball of radius~$N$ centered on the neutral element in the Cayley graph of~$G$ as
$$
\mathbf{B}(N) = \ \bigcup_{k=0}^N \ \bigcup_{\eta=-N}^{N-k} \ G^{\eta, k}
$$
By a calculation which parallels that in~\cite[Proof of Theorem 5.1]{CarusoWiestGeneric2},
we deduce that the proportion of elements, among those of $\mathbf{B}(N)$, admitting a conjugation to a rigid, loxodromically-acting element, tends to~1 exponentially quickly.
\end{proof}

{\bf Acknowledgements } The first author was supported by the ``initiation to research" project no.11140090 from Fondecyt and by Dr. Andr\'es Navas through the Project USA1555 from University of Santiago de Chile.  He also acknowledges support by PIA-CONICYT ACT1415 and by MTM2010-19355 and FEDER. 

\end{document}